\documentclass[12pt]{amsart}
\usepackage{amssymb,amscd}
\usepackage{verbatim}

\let\frak\mathfrak

\def\>{\relax\ifmmode\mskip.666667\thinmuskip\relax\else\kern.111111em\fi}
\def\<{\relax\ifmmode\mskip-.333333\thinmuskip\relax\else\kern-.0555556em\fi}
\def\vsk#1>{\vskip#1\baselineskip}
\def\vv#1>{\vadjust{\vsk#1>}\ignorespaces}
\def\vvn#1>{\vadjust{\nobreak\vsk#1>\nobreak}\ignorespaces}
 \let\alb\allowbreak
\def\fratop{\genfrac{}{}{0pt}1}
\def\satop#1#2{\fratop{\scriptstyle#1}{\scriptstyle#2}}
  \let\ssize\scriptstyle
\let\sssize\scriptscriptstyle

\let\Medskip\medskip
\def\medskip{\par\Medskip}
\let\Bigskip\bigskip
\def\bigskip{\par\Bigskip}

\let\Maketitle\maketitle
\def\maketitle{\Maketitle\thispagestyle{empty}\let\maketitle\empty}

\newtheorem{thm}{Theorem}[section]
\newtheorem{cor}[thm]{Corollary}
\newtheorem{lem}[thm]{Lemma}
\newtheorem{prop}[thm]{Proposition}

\numberwithin{equation}{section}

\theoremstyle{definition}
\newtheorem*{rem}{Remark}
\newtheorem*{example}{Example}

\let\mc\mathcal
\let\nc\newcommand

\let\al\alpha

\let\dl\delta
\let\Dl\Delta
\let\eps\varepsilon

\let\la\lambda

\let\pho\phi
\let\phi\varphi
\let\si\sigma

\let\der\partial
\let\Hat\widehat

\let\ox\otimes
\let\Tilde\widetilde
\let\bra\langle
\let\ket\rangle

\let\ge\geqslant
\let\geq\geqslant
\let\le\leqslant
\let\leq\leqslant

\let\on\operatorname
\let\bi\bibitem
\let\bs\boldsymbol

\def\C{{\mathbb C}}
\def\Z{{\mathbb Z}}
\def\R{{\mathbb R}}
\def\PP{{\mathbb P}}

\def\F{{\mc F}}

\def\+#1{^{\{#1\}}}
\def\lsym#1{#1\alb\dots\relax#1\alb}
\def\lc{\lsym,}

\def\End{\on{End}}

\def\Hom{\on{Hom}}
\def\id{\on{id}}

\def\ii{i,\<\>i}
\def\ij{i,\<\>j}
\def\ik{i,\<\>k}
\def\il{i,\<\>l}
\def\ji{j,\<\>i}
\def\jj{j,\<\>j}
\def\jk{j,\<\>k}
\def\kj{k,\<\>j}
\def\kl{k,\<\>l}

\def\II{I\<\<,\<\>I}

\def\ioi{i+1,\<\>i}

\def\ppo{p,\<\>p+1}
\def\pop{p+1,\<\>p}
\def\pci{p,\<\>i}
\def\pcj{p,\<\>j}
\def\poi{p+1,\<\>i}
\def\poj{p+1,\<\>j}

\def\gl{\mathfrak{gl}}
\def\gln{\mathfrak{gl}_N}

\def\Ugln{U(\gln)}

\def\Yn{Y\<(\gln)}

\def\beq{\begin{equation}}
\def\eeq{\end{equation}}
\def\be{\begin{equation*}}
\def\ee{\end{equation*}}

\nc{\bea}{\begin{eqnarray*}}
\nc{\eea}{\end{eqnarray*}}
\nc{\bean}{\begin{eqnarray}}
\nc{\eean}{\end{eqnarray}}
\nc{\Ref}[1]{{\rm(\ref{#1})}}

\let\ga\gamma
\let\Ga\Gamma

\nc{\Il}{{\mc I_{\bs\la}}}
\nc{\bla}{{\bs\la}}
\nc{\Fla}{\F_\bla}
\nc{\tfl}{{T^*\Fla}}
\nc{\GL}{{GL_n(\C)}}
\nc{\GLC}{{GL_n(\C)\times\C^*}}

\let\aal\al 

\def\zzz{z_1\lc z_n}

\def\Czh{\C[\zzz,h]}

\def\ty{\Tilde Y\<(\gln)}

\def\IMI{{I^{\<\>\min}}}

\let\sd s 

\def\st{\tilde s}

\def\ib{\bs i}
\def\jb{\bs j}

\def\iib{\ib,\<\>\ib}
\def\ijb{\ib,\<\>\jb}

\def\jib{\jb,\<\>\ib}

\def\zb{\bs z}

\def\zzii{z_1\lc z_{i+1},z_i\lc z_n}

\def\ip{\<\>i\>\prime}
\def\ipi{\>\prime\<\>i}

\def\Xin{X^\infty}

\def\Xk{X^\kk}

\def\ddk_#1{\kk_{#1}\<\>\frac\der{\der\<\>\kk_{#1}}}

\def\bul{\mathbin{\raise.2ex\hbox{$\sssize\bullet$}}}
\def\intt{\mathchoice
{\mathop{\raise.2ex\rlap{$\,\,\ssize\backslash$}{\intop}}\nolimits}
{\mathop{\raise.3ex\rlap{$\,\sssize\backslash$}{\intop}}\nolimits}
{\mathop{\raise.1ex\rlap{$\sssize\>\backslash$}{\intop}}\nolimits}
{\mathop{\rlap{$\sssize\<\>\backslash$}{\intop}}\nolimits}}

\let\kk q 
\let\kp\kappa 
\let\cc c
\def\kkk{\kk_1\lc\kk_N}

\let\Ko K
\def\Kh{\Hat\Ko}
\def\zzip{z_1\lc z_{i-1},z_i\<-\kp,z_{i+1}\lc z_n}

\def\GZ/{Gelfand-Zetlin}
\def\KZ/{{\slshape KZ\/}}
\def\qKZ/{{\slshape qKZ\/}}
\def\XXX/{{\slshape XXX\/}}

\def\zz{{\bs z}}
\def\qq{{\bs q}}
\def\TT{{\bs t}}

\def\Sym{\on{Sym}}

\nc{\A}{{\mc C}}

\def\GG{{\bs \Ga}}
\def\II{{\mc I}}
\def\XX{{\mc X}}
\def\a{{\frak{a}}}
\def\CC{{\frak{C}}}
\def\St{{\on{Stab}}}

\def\Czh{{(\C^N)^{\otimes n}\otimes\C(\zz;h)}}

\begin{document}

\hrule width0pt
\vsk->

\title[Partial flag varieties, stable envelopes and weight functions]
{Partial flag varieties, stable envelopes \\ and weight functions}

\author
[R.\,Rim\'anyi, V.\,Tarasov, A.\,Varchenko]
{ R.\,Rim\'anyi$\>^{\star}$,
V.\,Tarasov$\>^\circ$, A.\,Varchenko$\>^\diamond$}

\maketitle

\begin{center}
{\it $^{\star\,\diamond}\<$Department of Mathematics, University
of North Carolina at Chapel Hill\\ Chapel Hill, NC 27599-3250, USA\/}

\vsk.5>
{\it $\kern-.4em^\circ\<$Department of Mathematical Sciences,
Indiana University\,--\>Purdue University Indianapolis\kern-.4em\\
402 North Blackford St, Indianapolis, IN 46202-3216, USA\/}

\vsk.5>
{\it $^\circ\<$St.\,Petersburg Branch of Steklov Mathematical Institute\\
Fontanka 27, St.\,Petersburg, 191023, Russia\/}
\end{center}

{\let\thefootnote\relax
\footnotetext{\vsk-.8>\noindent
$^\star\<${\sl E\>-mail}:\enspace rimanyi@email.unc.edu\>,
supported in part by NSF grant DMS-1200685\\
$^\circ\<${\sl E\>-mail}:\enspace vt@math.iupui.edu\>, vt@pdmi.ras.ru\>,
supported in part by NSF grant DMS-0901616\\
$^\diamond\<${\sl E\>-mail}:\enspace anv@email.unc.edu\>,
supported in part by NSF grant DMS-1101508}}

\begin{abstract}
We consider the cotangent bundle $T^*\F_\bla$ of a $GL_n$ partial flag variety,
$\bla= (\la_1,\dots,\la_N), |\bla|=\sum_i\la_i=n$,
and the torus $T=(\C^\times)^{n+1}$ equivariant cohomology $H^*_T(T^*\F_\bla)$. In \cite{MO},
a Yangian module structure was introduced on $\oplus_{|\bla|=n}
H^*_T(T^*\F_\bla)$. We identify this Yangian module structure with the Yangian module structure
introduced in \cite{GRTV}. This identifies the operators of quantum multiplication by divisors on
$H^*_T(T^*\F_\bla)$, described in \cite{MO}, with the action of the dynamical Hamiltonians
from \cite{TV2, MTV1, GRTV}. To construct these identifications we provide a formula for the stable envelope maps,
associated with the partial flag varieties and introduced in \cite{MO}. The formula is in terms of the Yangian weight functions
introduced in \cite{TV1}, c.f. \cite{TV3, TV4}, in order to construct q-hypergeometric solutions of \qKZ/ equations.

\end{abstract}

\bigskip

\begin{center}
{\it To the memory of I.M.\,Gelfand (1913-2009)}
\end{center}

\bigskip

{\small \tableofcontents }

\setcounter{footnote}{0}
\renewcommand{\thefootnote}{\arabic{footnote}}

\section{Introduction}
In \cite{MO}, D.\,Maulik and A.\,Okounkov develop a general theory connecting quantum groups and equivariant
quantum cohomology of Nakajima quiver varieties, see \cite{N1, N2}. In this paper, we consider the constructions
and results of that general theory applied to the cotangent bundles of $GL_n$ partial flag varieties. We identify the objects and results
from \cite{MO}
with known objects and results associated with the Yangian $Y(\gln)$, see \cite{TV1, TV2, TV3, TV4, MTV1, GRTV}.

More precisely, we consider the cotangent bundle $T^*\F_\bla$ of a $GL_n$ partial flag variety,
$\bla= (\la_1,\dots,\la_N), |\bla|=\sum_i\la_i=n$,
and the torus $T=(\C^\times)^{n+1}$ equivariant cohomology $H^*_T(T^*\F_\bla)$. In \cite{MO},
a Yangian module structure was introduced on $\oplus_{|\bla|=n}
H^*_T(T^*\F_\bla)$. We identify this Yangian module structure with the Yangian module structure
introduced in \cite{GRTV}. This identifies the operators of quantum multiplication by divisors on
$H^*_T(T^*\F_\bla)$, described in \cite{MO}, with the action of the dynamical Hamiltonians
from \cite{TV2, MTV1, GRTV}. To construct these identifications we provide a formula for the stable envelope maps,
associated with the partial flag varieties and introduced in \cite{MO}. The formula is in terms of the Yangian weight functions
introduced in \cite{TV1}, c.f. \cite{TV3, TV4}, in order to construct q-hypergeometric solutions of \qKZ/ equations.

\smallskip
In Section \ref{sec EQ}, we follow \cite{MO} and define the stable envelope maps associated with partial flag varieties,
$\St_\si : (\C^N)^{\otimes n}\otimes\C[z_1,\dots,z_n;h] \to H^*_T(T^*\F_\bla)$, $\si\in S_n$.
In Section \ref{sec Weight functions}, we introduce weight functions and
the cohomological weight function maps
$[W_\si] : (\C^N)^{\otimes n}\otimes\C[z_1,\dots,z_n;h] \to H^*_T(T^*\F_\bla)$, $\si\in S_n$.
In Section \ref{sec Stable envelope maps and weight function maps} we prove
our main result, Theorem \ref{main thm}, which relates the stable envelope maps and
the cohomological weight function maps,
\bea
[W_\si] = c_\bla \circ\on{Stab}_\sigma,
\eea
where $c_\bla$ is the operator of multiplication by an element $c_\bla(\bs\Theta)\in H^*_T(T^*\F_\bla)$ defined in
\Ref{c bla}. The element $c_\bla(\bs\Theta)$ is not a zero-divisor.
The inverse maps to the stable envelope maps for  $GL_n$ partial flag varieties were considered in \cite[Formulas (5.9), (5.10)]{GRTV}.
One of those maps $\nu$ was reintroduced in Section \ref{pr st yang} where we prove that
\bea
\nu\circ\St_{\on{id}} = \on{Id},
\eea
here $\on{id}\in S_n$ is the identity.

In Section \ref{sec: ortho}, we describe the orthogonality relations for the stable envelope maps
$\St_{\on{id}}$ and $\St_{\si_0}$, where $ \si_0\in S_n$ is the longest permutation.
The orthogonality relations are analogues of the orthogonality relations for Schubert cycles corresponding
to dual Young diagrams.

In Section \ref{sec Yang}, our Theorem \ref{thm stab yang} and Corollary \ref{Y MO} say that the Yangian module structure on
$\oplus_{|\bla|=n} H^*_T(T^*\F_\bla)$, introduced in \cite{MO},
coincides with the Yangian structure introduced in \cite{GRTV} and Section \ref{sec cohom and Yang}.
We introduce the dynamical Hamiltonians and trigonometric dynamical connection in Section \ref{Bethe subalgebras sec}.
In Corollary \ref{cor qm = dh}, we identify the operators of quantum multiplication by divisors $D_i, i=1,\dots,N$, on
$H^*_T(T^*\F_\bla)$ with the action of the dynamical Hamiltonian $X^q_{\bla,i}, i=1,\dots,N$.
This identifies the quantum connection on $H^*_T(T^*\F_\bla)$ with the trigonometric dynamical connection.
In Section \ref{Bethe subalgebras sec}, we discuss also the \qKZ/ difference connection, compatible with the trigonometric dynamical connection, see \cite{TV2}.
The \qKZ/ difference connection on $H^*_T(T^*\F_\bla)$
corresponds to the shift operator difference connection introduced in \cite{MO}.

\smallskip
This paper is motivated by two goals. The first is to relate the stable envelope maps and the cohomological weight function
maps. The second goal is to identifies the quantum connection on $H^*_T(T^*\F_\bla)$ with the trigonometric dynamical connection.
The flat sections of the trigonometric dynamical connection and the associated \qKZ/ difference connection were constructed in
\cite{SV, MV, TV2} in the form of multidimensional hypergeometric integrals. The results of this paper allow us to construct
flat sections of the quantum connection and the shift operator difference connection on $H^*_T(T^*\F_\bla)$ in the form of
multidimensional hypergeometric integrals. Such a presentation of flat sections manifests the Landau-Ginzburg mirror
symmetry for the cotangent bundles of partial flag varieties.

\smallskip
The authors thank D. Maulik and A. Okounkov for answering questions on \cite{MO}.

\section{Stable envelopes}
\label{sec EQ}
In this section we follow \cite{MO} and define stable envelopes associated with partial flag varieties.

\subsection{Partial flag varieties}
\label{sec Partial flag varieties}
Fix natural numbers $N, n$. Let \,$\bla\in\Z^N_{\geq 0}$, \,$|\bla|=\la_1+\dots+\la_N =n$.
Consider the partial flag variety
\;$\Fla$ parametrizing chains of subspaces
\be
0\,=\,F_0\subset F_1\lsym\subset F_N =\,\C^n
\ee
with \;$\dim F_i/F_{i-1}=\la_i$, \;$i=1\lc N$.
Denote by \,$\tfl$ the cotangent bundle of \;$\Fla$, and let $\pi: \tfl \to \Fla$ be the projection of the bundle.
Denote
\bea
\XX_n = \cup_{|\bla|=n} T^*\F_\bla .
\eea
\begin{example}
If $n=1$, then $\bla=(0,\dots,0,1_i,0,\dots,0)$, $T^*\F_\bla$ is a point and $\XX_1$ is the union of $N$ points.

If $n=2$ then $\bla= (0,\dots,0,1_i,0, \dots,0,1_j,0,\dots,0)$ or $\bla= (0,\dots,0,2_i,0,\dots,0)$.
In the first case $T^*\F_\bla$ is the cotangent bundle of projective line, in the second case $T^*\F_\bla$ is a point.
Thus $\XX_2$ is the union of $N$ points and $N(N-1)/2$ copies of the cotangent bundle of projective line.

\end{example}

Let $I=(I_1\lc I_N)$ be a partition of $\{1\lc n\}$ into disjoint subsets
$I_1\lc I_N$. Denote $\Il$ the set of all partitions $I$ with
$|I_j|=\la_j$, \;$j=1,\dots N$.

Let $u_1,\dots,u_n$ be the standard basis of $\C^n$.
For any $I\in\Il$, let $x_I\in \F_\bla$ be the point corresponding to the coordinate flag
$F_1\subset\dots\subset F_N$, where $F_i$
\,is the span of the standard basis vectors \;$u_j\in\C^n$ with
\,$j\in I_1\lsym\cup I_i$. We embed $\F_\bla$ in $T^*\F_\bla$ as the zero section and consider
the points $x_I$ as points of $T^*\F_\bla$.

\subsection{Schubert cells}

For any $\sigma\in S_n$, we consider the
coordinate flag in $\C^n$,
\bea
V^{\sigma}\ : \ 0\,=\,V_0 \subset V_1\subset\dots\subset V_n=\,\C^n
\eea
where $V_i$ is the span of $u_{\sigma(1)},\dots,u_{\sigma(i)}$. For $I\in\Il$ we define the {\it Schubert cell}
\bea
\Omega_{\sigma,I}=\{ F\in \F_\bla\ |
\dim (F_p\cap V^\sigma_q) = \#\{ i \in I_1\cup \ldots \cup I_p \ |\ \sigma^{-1}(i)\leq q\} \ \forall p\leq N, \forall q\leq n \}.
\eea

\begin{lem}
The Schubert cell $\Omega_{\sigma,I}$ is an affine space of dimension
\bea
\ell_{\sigma,I}=\#\{(i,j)\in \{1,\ldots,n\}^2\ |\ \sigma(i)\in I_a, \sigma(j)\in I_b, a<b, i>j\}.
\eea
For a fixed $\sigma$ the flag manifold is the disjoint union of the cells $\Omega_{\sigma,I}$. We have $x_I\in \Omega_{\sigma,I}$.
\end{lem}

\begin{proof}
The structure of Schubert cells is well known, see e.g. \cite[Sect.2.2]{fp}.
\end{proof}

\subsection{Equivariant cohomology}
\label{sec:equiv}

Denote $G=GL_n(\C)\times \C^\times$.
Let \,$A\!\subset GL_n(\C)$ \,be the torus of diagonal matrices.
Denote $T=A\times\C^\times$ the subgroup of $G$.

The groups \,$A\!\subset GL_n$ act on \;$\C^n\<$ and hence on \,$\tfl$.
Let the group \;$\C^\times\<$ act on \,$\tfl$ by multiplication in each fiber.
We denote by $-h$ its $\C^\times$-weight.

\vsk.2>
We consider the equivariant cohomology algebras $H^*_{T}(\tfl;\C)$ and
\bea
H_T^*(\XX_n)=
\oplus_{|\bla|=n} H^*_{T}(\tfl;\C).
\eea
Denote by $\Ga_i=\{\ga_{i,1}\lc\ga_{i,\la_i}\}$ the set of
the Chern roots of the bundle over $\Fla$ with fiber $F_i/F_{i-1}$.
Let \;$\GG=(\Ga_1\<\>\lsym;\Ga_N)$. Denote by $\zb=\{\zzz\}$ the Chern roots
corresponding to the factors of the torus $T$.
Then
\vvn-.2>
\bea
\label{Hrel}
H^*_{T}(\tfl)
\, = \,\C[\GG]^{S_{\la_1}\times\dots\times S_{\la_N}}
\otimes \C[\zz]\otimes\C[h]
\>\Big/\Bigl\bra\,
\prod_{i=1}^N\prod_{j=1}^{\la_i}\,(u-\ga_{\ij})\,=\,\prod_{a=1}^n\,(u-z_a)
\Bigr\ket\,.
\eea
The cohomology
$H^*_{T}(\tfl)$ is a module over $H^*_{T}({pt};\C)=\C[\zz]\otimes\C[h]$.

\begin{example}
If $n=1$, then
\bea
H_T^*(\XX_1) = \oplus_{i=1}^N H_T^*(T^*\F_{(0,\dots,0,1_i,0,\dots,0)})
\eea
is naturally isomorphic to $\C^N\otimes\C[z_1;h]$ with basis
$v_i=(0,\dots,0,1_i,0,\dots,0)$, $i=1,\dots,N$.
\end{example}

For $i=1,\dots,N$, denote $\la^{(i)}=\la_1+\dots+\la_i$. Denote $\Theta_i=\{\theta_{i,1}, \dots, \theta_{i,\la^{(i)}}\}$
the Chern roots of the bundle over $\F_\bla$ with fiber $F_i$. Let $\bs \Theta=(\Theta_1,\dots,\Theta_{N})$. The relations
\bea
\prod_{j=1}^{\la^{(i)}}(u-\theta_{i,j}) = \prod_{\ell=1}^i\prod_{j=1}^{\la_i}\,(u-\ga_{\ij}), \qquad i=1,\dots,N,
\eea
define the homomorphism
\bea
\C[\bs \Theta]^{S_{\la^{(1)}}\times\dots\times S_{\la^{(N)}}}
\otimes \C[\zz]\otimes\C[h]\to H^*_{T}(\tfl).
\eea

\subsection{Fixed point sets}

\vsk.2>
The set \,$(\tfl)^{A}\!$ of fixed points of the torus $A$ action is $(x_I)_{I\in\Il}$. We have
\bea
(\XX_n)^A = \XX_1\times\dots\times\XX_1.
\eea
The cohomology algebra $H_T^*((\XX_n)^A)$ is naturally isomorphic to
\bea
(\C^N)^{\otimes n}\otimes \C[\zz;h].
\eea
This isomorphism sends the identity element $1_I\in H^*_T(x_I)$ to the vector
\bean
\label{v_I}
v_I = v_{i_1}\otimes\dots\otimes v_{i_n},
\eean
where $i_j =i$ if $i_j\in I_i$.

\subsection{Chamber decomposition}
Let $\a$ be the Lie algebra of $A$. The cocharacters $\eta :\C^\times \to A$ form a lattice of rank $n$.
We define
\bea
\a_\R = \on{Cochar}(A)\otimes_\Z\R\subset\a .
\eea
Each weight of $A$ defines a hyperplane in $\a_\R$.

Let $z_1,\dots,z_n$ be the standard basis of the dual space $\a^*$, as in Section \ref{sec:equiv}. Then the
{\it torus roots} are the $A$-weights $\al_{i,j}=z_i-z_j$ for all $i\ne j$.
The root hyperplanes partition $\a_\R$ into open chambers
\bea
\a_\R - \cup\al_{i,j}^\perp = \cup_{\sigma\in S_n} \CC_\sigma.
\eea
The chamber $\CC_\sigma$ consists of points $p\in\a_\R$ such that
$z_{\sigma(1)}(p)>\dots>z_{\sigma(n)}(p)$.

\subsection{Stable leaves}
\label{sec: St Leaves}

Let $\CC$ be a chamber. We say that $x\in\XX_n$ is $\CC$-stable
if the limit $\lim_{z\to 0}\eta(z)\cdot x \in (\XX_n)^A$ exists for one (equivalently, all) cocharacters
$\eta\in\CC$. This limit is independent of the choice of $\eta\in\CC$ and will be denoted by $\lim_\CC x$.

Given a point $x_I\in(\XX_n)^A$, we denote by $\on{Leaf}_{\CC,I} = \{x\ |\ \lim_\CC x = x_I\}$ the stable leaf of $x_I$.
For $\sigma\in S_n$, $I\in\Il$, we denote by $C\Omega_{\sigma,I} \subset T^*\F_\bla$ the conormal bundle of the Schubert cell
$\Omega_{\sigma,I}$.

\begin{lem} \label{lem:leaf_conormal}
We have $\on{Leaf}_{\CC_\sigma,I}= C\Omega_{\sigma,I}$.
\end{lem}

\begin{proof} Consider the natural $A$-invariant identification of
$\pi^{-1}(\Omega_{\sigma,I})$ with $\C^{\ell_{\sigma,I}} \oplus \C^{\dim \F_\bla}$
mapping $x_I$ to the origin. The weights on the first component $\C^{\ell_{\sigma,I}}$ are
\begin{equation} \label{eqn:w_S}
z_{\sigma(j)}-z_{\sigma(i)}\qquad\text{for} \qquad \sigma(i)\in I_a, \sigma(j)\in I_b, a<b, i>j,
\end{equation}
and the weights on the second component $\C^{\dim \F_\bla}$ are
\begin{equation} \label{eqn:w_T}
z_{\sigma(i)}-z_{\sigma(j)}\qquad\text{for} \qquad \sigma(i)\in I_a, \sigma(j)\in I_b, a<b.
\end{equation}
Consider the splitting
$$\C^{\ell_{\sigma,I}} \oplus \C^{\dim \F_\bla} = \underbrace{\C^{\ell_{\sigma,I}} \oplus T^*_1}_{C\Omega_{\sigma,I}} \oplus T^*_2,$$
where $T^*_1$ is the sum of weight subspaces with weights
\begin{equation} \label{eqn:w_T1}
z_{\sigma(i)}-z_{\sigma(j)}\qquad\text{for} \qquad \sigma(i)\in I_a, \sigma(j)\in I_b, a<b, i<j,
\end{equation}
and $T^*_2$ is the sum of weight subspaces with weights
\begin{equation} \label{eqn:w_T2}
z_{\sigma(i)}-z_{\sigma(j)}\qquad\text{for} \qquad \sigma(i)\in I_a, \sigma(j)\in I_b, a<b, i>j.
\end{equation}
For a cocharacter $\eta\in \CC_\sigma$ the weights in (\ref{eqn:w_S}) and (\ref{eqn:w_T1})
are all positive and the weights in (\ref{eqn:w_T2}) are all negative.

Therefore, a point in $\pi^{-1}(\Omega_{\sigma,I})$ has $\lim_{\CC} x=x_I$ if it belongs to $C\Omega_{\sigma,I}$ and
is not $\CC$-stable if it does not belong to $C\Omega_{\sigma,I}$. Applying the same argument for all other
$J\in \II_\lambda$, we see that $\on{Leaf}_{\CC_\sigma,I} \subset \pi^{-1}(\Omega_{\sigma,I})$,
and moreover, $\on{Leaf}_{\CC_\sigma,I}=C\Omega_{\sigma,I}$.
\end{proof}

For $\sigma\in S_n$, we define the {\it geometric} partial ordering
on the set $\Il$.
For $I,J\in\Il$, we say that $J\leq_g I$ if $x_J$ lies in the closure of $\on{Leaf}_{\sigma,I}$.

We also define the {\it combinatorial} partial ordering.
For $I,J\in\Il$, let
\bea
\sigma^{-1}(\cup_{\ell=1}^k I_\ell) =\{a_1^k<\dots< a_{\la^{(k)}}^k\},
\qquad
\sigma^{-1}(\cup_{\ell=1}^k J_\ell) =\{b_1^k<\dots< b_{\la^{(k)}}^k\}
\eea
for $k=1,\dots,N-1$.
We say that $J\leq_c I$ if $b_i^k\leq a_i^k$ for $k=1,\dots,N-1$, $i=1,\dots,\la^{(k)}$.

\begin{lem}
The geometric and combinatorial partial orderings are the same.
\end{lem}

\begin{proof}
This is the so-called ``Tableau Criterion'' for the Bruhat (i.e. geometric) order, see e.g. \cite[Thm. 2.6.3]{BB}.
\end{proof}

In what follows we will denote both partial orderings by $\leq_\sigma$.

\begin{lem}
\label{lem incomp}
For $I,J\in\Il$, $I\ne J$, there exists $\si\in S_n$ such that $J\not\leq_\si I$.
\end{lem}

\begin{proof}
The group $S_n$ has an obvious action on $\II_\lambda$ as well. Observe that $J\leq_\sigma I$ is equivalent to
$ \si^{-1}(J) \le_{\id} \si^{-1}(I)$.
Hence the requirement of the Lemma is achieved by choosing $\sigma$ such that $\sigma^{-1}(I)$ is the $\leq_{\on{id}}$-smallest
element of $\II_\lambda$, namely $(\{1,\ldots,\lambda^{(1)}\},\{\lambda^{(1)}+1,\ldots,\lambda^{(2)}\},\ldots)$.
\end{proof}

\medskip
For $\sigma\in S_n$, $I\in\Il$, we define $\on{Slope}_{\sigma,I}=\cup_{J\leq_\sigma I}\on{Leaf}_{\sigma,J}$.
The $\on{Slope}_{\sigma,I}$ is a closed subset of $T^*\F_\bla$ by \cite[Lemma 3.2.7]{MO}.

\subsection{Stable envelopes}
Given a closed $T$-invariant subset $Y\subset T^*\F_\bla$ and a class $E\in H^*_T(T^*\F_\bla)$,
we say that $E$ is supported in $Y$ if $E|_{T^*\F_\bla-Y} =0.$

For given $I$ and $\sigma$, we define the following classes in $H^*_T(pt)$
\bea
e^{hor}_{\sigma,I,+}=\prod_{a<b} \mathop{\prod_{\sigma(i)\in I_a}}_{\sigma(j)\in I_b} \prod_{i>j} (z_{\sigma(j)}-z_{\sigma(i)}), \qquad
e^{hor}_{\sigma,I,-}=\prod_{a<b} \mathop{\prod_{\sigma(i)\in I_a}}_{\sigma(j)\in I_b} \prod_{i<j} (z_{\sigma(j)}-z_{\sigma(i)}),
\eea
\bea
e^{ver}_{\sigma,I,+}=\prod_{a<b} \mathop{\prod_{\sigma(i)\in I_a}}_{\sigma(j)\in I_b} \prod_{i<j} (z_{\sigma(i)}-z_{\sigma(j)}-h), \ \
e^{ver}_{\sigma,I,-}=\prod_{a<b} \mathop{\prod_{\sigma(i)\in I_a}}_{\sigma(j)\in I_b} \prod_{i>j} (z_{\sigma(i)}-z_{\sigma(j)}-h).
\eea
These are the products of the positive (``$+$'') and negative (``$-$'') $T$-weights (w.r.t. to $\CC_{\sigma}$)
at $x_I$ in the tangent to $\F_\bla$ direction (``hor'') and fiber direction (``ver'').

Let $e_{\sigma,I,-}=e^{hor}_{\sigma,I,-} \cdot e^{ver}_{\sigma,I,-}$. Let $\on{sgn}_{\sigma,I}
=
(-1)^{\on{codim}(\Omega_{\sigma,I}\subset \F_\bla)}=
\deg(e^{hor}_{\sigma,I,-})$.

\begin{thm}
\label{thm:MO_stab}
For any $\sigma\in S_n$, there exists a unique map of $H^*_T(pt)$-modules
\bea
\on{Stab}_\sigma : H^*_T((\XX_n)^A) \to H^*_T(\XX_n)
\eea
such that for any $\bla$ with $|\bla|=n$ and any $I\in \Il$, the stable envelope $E_{\sigma,I}=\on{Stab}_\sigma(1_I)$ satisfies:
\begin{enumerate}
\item[(i)] $\on{supp}\,E_{\sigma,I}\subset \on{Slope}_{\sigma,I}$\,,

\item[(ii)] $E_{\sigma,I}|_{x_I} = \on{sgn}_{\sigma,I} \cdot e_{\sigma,I,-}$\,,

\item[(iii)] $\deg_\zz E_{\sigma,I}|_{x_J} < \dim \F_\bla = \sum_{1\leq i<j\leq N}\la_i\la_j$ for any $J\in\Il$ with $J<_\sigma I$.

\end{enumerate}

\end{thm}

This is Theorem 3.3.4 in \cite{MO} applied to $GL_n$ partial flag varieties. The choice of sign in (ii) is
called a polarization in \cite{MO}. We will fix the polarization $\on{sgn}_{\sigma,I}$ in the whole paper.

\subsection{Geometric $R$-matrices}

The maps $\on{Stab}_\sigma$ become isomorphisms after inverting the
elements $(e_{\sigma,I,-})_{I\in\Il}$.
For $\sigma',\sigma\in S_n$, we define the $R$-matrix
\bean
\label{R}
\phantom{aaaa}
R_{\si',\si} = \St_{\si'}^{-1}\circ \St_{\si} \in \End(H_T((\XX_n)^A))\otimes \C(\zz;h)=\End((\C^N)^{\otimes n})\otimes \C(\zz;h),
\eean
where $\C(\zz;h)$ is the algebra of rational functions in $\zz, h$.

\medskip
\noindent
{\bf Example}\ \cite[Example 4.1.2]{MO}.\enspace
Let $n=2$. The group $S_2$ consists of two elements:\ $\on{id}$ and the transposition $s$.
After the identification
$H_T^*((\XX_n)^A) = (\C^N)^{\otimes 2}\otimes \C[\zz;h]$, the $R$-matrix is given by
\bean
\label{RR}
R_{s,\on{id}} = R(z_1-z_2),
\eean
where we define
\bea
R(u) = \frac {u \on{Id} - h P}{u-h},
\eea
and $P\in\End(\C^N\otimes\C^N)$ is the permutation of tensor factors.

For the convenience of the reader we show the calculation leading to (\ref{RR}). The space $\XX_2$ is the union
of $N(N-1)/2$ copies of $T^*\PP^1$ and $N$ points.
The space $\XX_2^A$ thus has $N(N-1)$ points in the $N(N-1)/2$ copies of $T^*\PP^1$ together
with the $N$ isolated points of $\XX_2$. On $H_T^*$ of the isolated points
	of $\XX_2$ both sides of (\ref{RR}) act as identity.
Let $i<j$. Consider the $T^*\PP^1$ component corresponding to $\lambda=(0,\ldots,0,1_i,0,\ldots,0,1_j,0,\ldots,0)$.
Then
\bea
H^*_T(T^*\PP^1) = \C[\gamma_{i,1},\gamma_{j,1}]\otimes \C[z_1,z_2;h]/ \left< (u-\gamma_{i,1})(u-\gamma_{j,1})=(u-z_1)(u-z_2) \right>.
\eea
The two fixed points $x_I$ and $x_J$ in this $T^*\PP^1$ component are indexed by
$I=(I_1,\dots,I_N)$ such that $I_i=\{1\}$ and $I_j=\{2\}$ and $I_m=\varnothing$ for all other indices, and
$J=(I_1,\dots,I_N)$ such that $I_i=\{2\}$ and $I_j=\{1\}$ and $I_m=\varnothing$ for all other indices.
Let $F_I$ and $F_J$ denote the fibers over $x_I$ and $x_J$ in $T^*\PP^1$. We have
\bea
&
\St_{\on{id}} &\quad :\quad 1_{I}\mapsto -[F_I]= \gamma_{i,1} - z_2, \qquad\ \ \qquad 1_{J}\mapsto [\PP^1]+[F_I]= \gamma_{i,1}-z_1-h,
\\
&
\St_{s} &\quad :\quad 1_{I}\mapsto [\PP^1]+[F_J] = \gamma_{i,1} - z_2-h, \ \ 1_{J}\mapsto -[F_J]= \gamma_{i,1}-z_1.
\eea
Here the geometric statements (e.g. $\St_{\on{id}}(1_I)=-[F_I]$) can be checked by verifying the
conditions of Theorem \ref{thm:MO_stab},
and the calculation (e.g. $-[F_I]=\gamma_{i,1} - z_2$) can be verified by equivariant localization. Therefore we have
\begin{equation}
\begin{split}
\St_s^{-1}\circ \St_{\on{id}}\left( 1_{I} \right) & = \St_s^{-1}\left( \gamma_{i,1}-z_2\right) \\
& = \St_s^{-1}\left( \frac{ z_1 -z_2 }{z_1-z_2-h} (\gamma_{i,1}-z_2-h) + \frac{-h}{z_1-z_2-h}(\gamma_{i,1}-z_1) \right)\\
& = \frac{ z_1 -z_2 }{z_1-z_2-h} \cdot 1_{I} + \frac{-h}{z_1-z_2-h} \cdot 1_{J}.
\end{split}
\end{equation}
This, together with a similar calculation for $1_J$ proves the claim (\ref{RR}) for the fixed points in $T^*\PP^1$.

\medskip

It is enough to consider $R$-matrices corresponding to pairs of
chambers separated by a wall.
Such a pair has the following form. For $i=1,\dots,n-1$, let $s_i\in S_n$ be the transposition $(i,i+1)$.
Any chamber $\CC_\si = \{p\in \a_\R\ | z_{\si(1)}(p)>\dots > z_{\si(n)}(p)\}$
is separated by a wall from exactly $n-1$ chambers. They are
$\CC_{\si s_i} = \{p\in \a_\R\ |\ z_{\si(1)}(p)>\dots > z_{\si(i-1)}>z_{\si(i+1)}>z_{\si(i)}>z_{\si(i+2)}>\dots> z_{\si(n)}(p)\}$,
$i=1,\dots,n-1$.

\begin{thm} [Section 4.1.6 in \cite{MO}]
\label{sec 416}
\bean
R_{\si s_i,\si} = R^{(\si(i),\si(i+1))}(z_{\si(i)}-z_{\si(i+1)}) \in \End((\C^N)^{\otimes n})\otimes \C(\zz;h),
\eean
where the superscript means that the $R$-matrix of formula \Ref{RR} operates in the $\si(i)$-th and $\si(i+1)$-th tensor factors.

\end{thm}

\section{Weight functions}
\label{sec Weight functions}

\subsection{Weight functions $W_{I}$}
For $I\in \Il$, we define the weight functions $W_{I}(\TT;\zz;h)$, c.f. \cite{TV1, TV4}.

Recall \,$\bla=(\la_1\lc\la_N)$. Denote  \,$\la^{(i)}\>=\la_1\lsym+\la_i$ and
 \,$\la^{\{1\}}\<=\sum_{i=1}^{N-1}\la^{(i)}=$
 \linebreak $ \sum_{i=1}^{N-1}(N\<\<-i)\>\la_i$\>.
Recall $I=(I_1,\dots,I_N)$. Set
\;$\bigcup_{\>k=1}^{\,j}I_k=\>\{\>i^{(j)}_1\!\lsym<i^{(j)}_{\la^{(j)}}\}$\>. Consider the variables
\,$t^{(j)}_a$, \,$j=1\lc N$, \,$a=1\lc\la^{(j)}$,
where \,$t^{(N)}_a=z_a$, \,$a=1\lc n$\>. Denote $t^{(j)}=(t^{(j)}_k)_{k\leq\la^{(j)}}$ and \,$\TT=(t^{(1)}, \dots, t^{(N-1)})$.

\vsk.2>
The weight functions are
\vvn.4>
\beq
\label{hWI-}
W_I(\TT;\zb;h)\,=\,(-h)^{\>\la^{\{1\}}}\,
\Sym_{\>t^{(1)}_1\!\lc\,t^{(1)}_{\la^{(1)}}}\,\ldots\;
\Sym_{\>t^{(N-1)}_1\!\lc\,t^{(N-1)}_{\la^{(N-1)}}}U_I(\TT;\zb;h)\,,
\vv.3>
\eeq
\be
U_I(\TT;\zb;h)\,=\,\prod_{j=1}^{N-1}\,\prod_{a=1}^{\la^{(j)}}\,\biggl(
\prod_{\satop{c=1}{i^{(j+1)}_c\<<\>i^{(j)}_a}}^{\la^{(j+1)}}
\!\!(t^{(j)}_a\<\<-t^{(j+1)}_c-h)
\prod_{\satop{d=1}{i^{(j+1)}_d>\>i^{(j)}_a}}^{\la^{(j+1)}}
\!\!(t^{(j)}_a\<\<-t^{(j+1)}_d )\,\prod_{b=a+1}^{\la^{(j)}}
\frac{t^{(j)}_a\<\<-t^{(j)}_b\<\<-h}{t^{(j)}_a\<\<-t^{(j)}_b}\,\biggr)\,.
\ee
In these formulas for a function $f(t_1,\dots,t_k)$ of some variables we denote
\be
\Sym_{t_1,\dots,t_k}f(t_1,\dots,t_k) = \sum_{\sigma\in S_k}f(t_{\sigma_1},\dots,t_{\sigma_k}).
\ee

\begin{example}
Let $N=2$, $n=2$, $\bla=(1,1)$, $I=(\{1\},\{2\})$,
$J=(\{2\}, \{1\})$. Then
\bea
W_I(\TT;\zz;h)= -h\, (t^{(1)}_1\<\<-z_2),
\qquad
W_J(\TT;\zz;h)= -h\, (t^{(1)}_1\<\<-z_1-h).
\eea
\end{example}

\begin{example}
Let $N=2$, $n=3$, $\bla=(1,2)$, $I=(\{2\}, \{1,3\})$. Then
\bea
W_I(\TT;\zz;h)= -h\, (t^{(1)}_1\<\<-z_1-h)(t^{(1)}_1\<\<-z_3).
\eea
\end{example}

For a subset $A\subset\{1,\dots,n\}$, denote $\zz_A=(z_a)_{a\in A}$. For $I\in\Il$, denote
$\zz_I=(\zz_{I_1},\dots,\zz_{I_N})$. For
$f(t^{(1)},\dots,t^{(N)})\in\C[t^{(1)},\dots,t^{(N)}]^{S_{\la^{(1)}}\times\dots\times S_{\la^{(N)}}}$,
we define $f(\zz_I)$ by replacing $t^{(j)}$ with $\cup_{k=1}^j\zz_{I_k}$.
Denote
\beq
c_\bla(\zb_I)\,=\,
\prod_{a=1}^{N-1}\>\prod_{\ij\in \cup_{b=1}^a I_b\!}\>(z_i\<-z_j\<-h)\,.
\vv.1>
\eeq

\begin{lem}
\label{lem W_I}

${}$

\begin{enumerate}
\item[(i)]

For $I, J\in\Il$, the polynomial $W_I(\zz_J;\zz;h)$ is divisible by $c_\bla(\zz_J)$.

\item[(ii)]
For $I\in\Il$,
\bea
W_I(\zz_I;\zz;h) =
c_\bla(\zz_I) \ \prod_{a<b} \mathop{\prod_{i\in I_a}}_{j\in I_b}
\Big( \prod_{i<j} (z_i-z_j) \prod_{i>j} (z_i-z_j-h) \Big).
\eea
\item[(iii)]
For $I, J\in\Il$, $J<_{\on{id}}I$
\bea
\deg_\zz W_I(\zz_J;\zz;h)< \deg_\zz W_I(\zz_I;\zz;h;\zz_I).
\eea

\item[(iv)]

For $I, J\in\Il$, we have $W_I(\zz_J;\zz;h)=0$ unless $J\leq_{\on{id}}I$.

\end{enumerate}

\end{lem}

Lemma \ref{lem W_I} is proved in Section \ref{proofs}.

\subsection{Weight functions $W_{\si,I}$}

For $\si\in S_n$ and $I\in\Il$, we define
\bea
W_{\si,I}(\TT;\zz;h) = W_{\si^{-1}(I)}(\TT;z_{\si(1)},\dots,z_{\si(n)};h),
\eea
where $\si^{-1}(I)=(\si^{-1}(I_1),\dots,\si^{-1}(I_N))$.

\begin{example}
Let $N=2$, $n=2$, $\bla=(1,1)$, $I=(\{1\}, \{2\})$, $J=(\{2\},\{1\})$. Then
\bea
& W_{\on{id},I}(\TT;\zz;h)= -h\, (t^{(1)}_1\<\<-z_2),&
\qquad
W_{\on{id},J}(\TT;\zz;h)= -h\, (t^{(1)}_1\<\<-z_1-h),
\\
&W_{s,I}(\TT;\zz;h)= -h\, (t^{(1)}_1\<\<-z_2-h),&
\qquad
W_{s,J}(\TT;\zz;h)= -h\, (t^{(1)}_1\<\<-z_1).
\eea
\end{example}

\begin{lem}
\label{lem W I sigma}
For any $\si\in S_n$, we have the following statements:
\begin{enumerate}
\item[(i)]

For $I, J\in\Il$, the polynomial $W_{\si,I}(\zz_J;\zz;h)$ is divisible by $c_\bla(\zz_J)$.

\item[(ii)]
For $I\in\Il$,
\bea
W_{\si,I}(\zz_I;\zz;h) =
c_\bla(\zz_I) \ \prod_{a<b} \mathop{\prod_{\sigma(i)\in I_a}}_{\sigma(j)\in I_b}
\Big( \prod_{i<j} (z_{\sigma(i)}-z_{\sigma(j)}) \prod_{i>j} (z_{\sigma(i)}-z_{\sigma(j)}-h) \Big).
\eea
\item[(iii)]
For $I, J\in\Il$, $J<_{\si}I$,
\bea
\deg_\zz W_{\si,I}(\zz_J;\zz;h)< \deg_\zz W_{\si,I}(\zz_I;\zz;h).
\eea

\item[(iv)]

For $I, J\in\Il$, we have $W_{\si,I}(\zz_J;\zz;h)=0$ unless $J\leq_{\si}I$.

\end{enumerate}

\end{lem}

Lemma \ref{lem W I sigma} follows from Lemma \ref{lem W_I}.

\begin{lem}
\label{lem W si W}
For any $\si\in S_n$, $I\in \Il$, $i=1,\dots,n-1$, we have
\begin{equation}
W_{\sigma s_{i,i+1},I} =
\frac{z_{\sigma(i)}-z_{\sigma(i+1)}}{z_{\sigma(i)}-z_{\sigma(i+1)}+h}W_{\sigma,I} +
\frac h{z_{\sigma(i)}-z_{\sigma(i+1)}+h} W_{\sigma, s_{\sigma(i),\sigma(i+1)}(I)} ,
\end{equation}
where $s_{i,j}\in S_n$ is the transposition of $i$ and $j$.

\end{lem}

Lemma \ref{lem W si W} is proved in Section \ref{proofs}.

\smallskip

Let $\si_0\in S_n$ be the longest permutation, that is, $\si_0: i\mapsto n+1-i$, $i=1,\dots,n$.

\begin{lem}
\label{lem:orto}

For $J,K\in \Il$, we have
\bea
\sum_{I\in\Il} \frac{W_{\on{id},J}(\zz_I;\zz;h) W_{\si_0,K}(\zz_I;\zz;h)}
{R(\zb_I)\,Q(\zb_I)\>c_\bla(\zb_I)^2}\,=\,\dl_{J,K}\,,
\eea
where
\bea
\label{RQ}
R(\zb_I)\,=\!\prod_{1\le a<b\le N}\,\prod_{i\in I_a}\,\prod_{j\in I_b}\,
(z_i-z_j)\,,\qquad
Q(\zb_I)\,=\!\prod_{1\le a<b\le N}\,\prod_{i\in I_a}\,\prod_{j\in I_b}\,
(z_i-z_j-h)\,.
\vv.3>
\eea
\end{lem}
Lemma \ref{lem:orto} is proved in Section \ref{proofs}.

\subsection{$R$-matrices}

Consider $\C[\TT;\zz;h]\otimes_{\C[\zz;h]}\C(\zz;h)$ as a $\C(\zz;h)$-module.
Denote by $M_\bla$ the $\C(\zz;h)$-submodule generated by
the polynomials $(W_{I}(\zz;h;\TT))_{I\in\Il}$. Denote $M_n=\oplus_{|\bla|=n}M_\bla$.

\begin{lem}

The module $M_\bla$ is free of rank $|\Il|$ with the basis $(W_I)_{I\in\Il}$.
\end{lem}
\begin{proof}
The lemma follows from parts (ii) and (iv) of Lemma \ref {lem W_I}.
\end{proof}

\begin{lem}
For any $\si\in S_n$ the polynomials $(W_{\si,I})_{I\in\Il}$ form a basis of
$M_\bla$.

\end{lem}
\begin{proof}
The fact that $W_{\si,I}$ belongs to $M_\la$ follows from Lemma \ref{lem W si W} by
induction on the length of $\si$. The independence of $(W_{\si,I})_{I\in\Il}$
follows from parts (ii) and (iv) of Lemma \ref{lem W I sigma}. Alternatively, the independence of
$(W_{\si,I})_{I\in\Il}$ follows by induction from Lemma \ref{lem W si W} with $\si$ replaced
by $\si s_i$ (this is equivalent to inverting the formula in Lemma \ref{lem W si W}).

\end{proof}

For $\si\in S_n$, we define the {\it algebraic weight function map}
\bea
W_\si : (\C^N)^{\otimes n}\otimes \C(\zz;h) \to M_n, \quad
v_I \mapsto W_{\si,I}(\TT;\zz;h).
\eea
For $\si',\si\in S_n$, we define the $R$-matrix
\bea
\tilde R_{\si',\si} = W_{\si'}^{-1} \circ W_\si \in \End((\C^N)^{\otimes n})\otimes \C(\zz;h).
\eea

\begin{thm} [\cite{TV4}]
\label{thm tv3}

For any $\si\in S_n$ and $i=1,\dots,n-1$, we have
\bean
\tilde R_{\si s_i,\si} = R^{(\si(i),\si(i+1))}(z_{\si(i)}-z_{\si(i+1)}) \in \End((\C^N)^{\otimes n})\otimes \C(\zz;h),
\eean
where the superscript means that the $R$-matrix of formula
\Ref{RR} operates in the $\si(i)$-th and $\si(i+1)$-th tensor factors.

\end{thm}
\begin{proof}
The theorem follows from Lemma \ref{lem W si W}.
\end{proof}

For $\si\in S_n$, we define the {\it cohomological weight function map}
\bea
[W_\si] : H_T^*((\XX_n)^A) = (\C^N)^{\otimes n}\otimes \C[\zz;h] \to H_T^*(\XX_n) , \quad
v_I \mapsto [W_{\si,I}(\bs\Theta;\zz;h)],
\eea
where $W_{\si,I}(\bs\Theta;\zz;h)$ is the polynomial $W_{\si,I}(\TT;\zz;h)$ in which variables
$t^{(j)}_i$ are replaced with $\theta_{j,i}$ and $[W_{\si,I}(\bs\Theta;\zz;h)]$ is the cohomology class represented by
$W_{\si,I}(\bs\Theta;\zz;h)$.

Let $\on{Loc} : H^*_T(\XX_n) \to H^*_T((\XX_n)^A)$ be the localization map
$w \mapsto (w|_{x_I})_{x_I\in (\XX_n)^A}$.
According to Lemma \ref{lem W I sigma}, the composition $\on{Loc} \circ [W_\si]$ is upper triangular with respect
to the order $\leq_\si$. Therefore -- after tensoring with $\C(\zz;h)$ -- the map $\on{Loc} \circ [W_\si]$ is invertible.
We obtain that for $\si',\si\in S_n$, we can define the $R$-matrix
\bea
\bar R_{\si',\si} = [W_{\si'}]^{-1}\circ [W_\si] \in \End((\C^N)^{\otimes n})\otimes \C(\zz;h),
\eea
and that
$$
\bar R_{\si s_i,\si} = \tilde R_{\si s_i,\si}.
$$

\begin{cor}
\label{thm coh weight}
For any $\si\in S_n$ and $i=1,\dots,n-1$, we have
\bean
\bar R_{\si s_i,\si} = R^{(\si(i),\si(i+1))}(z_{\si(i)}-z_{\si(i+1)}) \in \End((\C^N)^{\otimes n})\otimes \C(\zz;h),
\eean
where the superscript means that the $R$-matrix of formula \Ref{RR} operates in the $\si(i)$-th and $\si(i+1)$-th tensor factors.
\end{cor}

\begin{rem}
The maps $ W_\si : (\C^N)^{\otimes n}\otimes \C(\zz;h) \to M_n,$\ $\si\in S_n$, form what is called in \cite{V}
local tensor coordinates on $M_n$. The stable envelope maps as well as cohomological weight function maps are basically other examples of that
 structure.
\end{rem}

\section{Stable envelope maps and weight function maps}
\label{sec Stable envelope maps and weight function maps}

\subsection{Main theorem}
Denote
\bean
\label{c bla}
c_\bla(\bs \Theta) = \prod_{a=1}^{N-1}\prod_{i=1}^{\la^{(a)}}\prod_{j=1}^{\la^{(a)}}(\theta_{a,i}-\theta_{a,j} - h) \in H^*_T(T^*\F_\bla).
\eean
Observe that $c_\bla(\bs \Theta)$ is the equivariant Euler class of the bundle $\oplus_{a=1}^{N-1} \Hom(F_a,F_a)$ if we make $\C^\times$ act on it with weight $-h$.
Here,  by a slight abuse of notation, we wrote $F_a$ for the bundle over $\F_\bla$ with fiber $F_a$.

Let $c_n(\bs \Theta) = (c_\bla(\bs \Theta))_{|\bla|=n} \in H^*_T(\XX_n)$. Note that $c_n(\bs \Theta)$ is not a zero-divisor in $H^*_T(\XX_n)$,
because none of its fixed point restrictions is zero.

\begin{thm}
\label{main thm}
For any $\si\in S_n$, the maps
$\on{Stab}_{\si} : H^*_T((\XX_n)^A) \to H^*_T(\XX_n)$ and
$[W_\si] : H_T^*((\XX_n)^A) \to H_T^*(\XX_n)$ are related:
\bean
\label{wfs}
[W_\si] = c_n \circ\on{Stab}_\sigma,
\eean
where
$c_n$ denotes the operator of multiplication by $c_n(\bs\Theta)$.

\end{thm}

Theorem \ref{main thm} is proved in Section \ref{sec trian}.

\begin{rem}
In \cite{RTV} we defined cohomology classes $\kappa_I$ and $\kappa'_I$ in the cohomology
ring of the cotangent bundle of a Grassmannian. The classes were suggested as candidates for the
equivariant fundamental cohomology class of the cotangent bundle of the Schubert
variety $\bar{\Omega}_{\on{id},I}$, which is singular in general.
Theorem \ref{main thm} implies that the classes $\kappa_I$, $\kappa'_I$ defined in \cite{RTV} are supported on $\on{Slope}_{\on{id},I}$.

\end{rem}

\subsection{Triangularity}
\label{sec trian}

The $\C(\zz;h)$-module $H^*_T(\XX_n)\otimes\C(\zz;h)$ has a basis $(w_I(\GG;\zz;h))_{I\in \Il, |\bla|=n}$,
\bea
w_I(\GG;\zz;h) = \prod_{1\leq i< j\leq N} \prod_{x\in \Gamma_i}
\prod_{k\in I_i} \prod_{\ell\in I_j} \frac{x-z_\ell}{z_k-z_\ell}.
\eea
We have
\bea
w_I(\zz_J;\zz;h) = \delta_{I,J}, \qquad I,J\in \Il.
\eea
Thus we have the distinguished basis $(w_I(\GG;\zz;h))_{I\in \Il, |\bla|=n}$ of the $\C(\zz;h)$-module
$H^*_T(\XX_n)\otimes\C(\zz;h)$ and the distinguished basis
$(1_I)_{I\in \Il, |\bla|=n}$ of the $\C(\zz;h)$-module $H^*_T((\XX_n)^A)$.

We denote the matrices of the $\C(\zz;h)$-module isomorphisms $\St_\si$ and $[W_\si]$ with respect to these bases
by $\mc A_\si$ and $\mc B_\si$ respectively.

\begin{lem}
The matrix $\mc A_\si$ is upper-triangular with respect to the partial ordering $\leq_\si$, that is,
if $(\mc A_\si)_{J,I} \ne 0$, then $J\leq_\si I$.

\end{lem}

\begin{proof}
The lemma follows from part (i) of Theorem \ref{thm:MO_stab}.
\end{proof}

\begin{lem}
The matrix $\mc B_\si$ is upper-triangular with respect to the partial ordering $\leq_\si$,
that is, if $(\mc B_\si)_{J,I} \ne 0$, then $J\leq_\si I$.

\end{lem}
\begin{proof}
The lemma follows from part (iv) of Lemma \ref{lem W I sigma}.
\end{proof}

Define the $\C(\zz;h)$-module isomorphism
\bea
f_\si =[W_\si]\circ (\St_\si)^{-1} : H^*_T(\XX_n)\otimes\C(\zz;h)\to H^*_T(\XX_n)\otimes\C(\zz;h).
\eea

\begin{cor}
\label{cor upper}
For any $\si\in S_n$, the matrix of $f_\si$ is upper-triangular with respect to the partial ordering $\leq_\si$.

\end{cor}

\begin{lem}
The operator $f_\si$ does not depend on $\si\in S_n$.

\end{lem}
\begin{proof} For $\si',\si\in S_n$ we have
\bea
R_{\si',\si} = [W_{\si'}]^{-1}\circ [W_\si] = \St_{\si'}^{-1}\circ (f_{\si'}^{-1}\circ f_\si )\circ \St_\si,
\qquad
R_{\si',\si} = \St_{\si'}^{-1}\circ\St_\si
\eea
Hence $f_{\si'}^{-1} \circ f_\si = 1$ and $f=f_\si$ does not depend on $\si$.
\end{proof}

\begin{lem}
The matrix of the operator $f$ is diagonal.

\end{lem}

\begin{proof}
For any $\si\in S_n$ the matrix of $f$ is upper-triangular by Corollary \ref{cor upper}. Lemma \ref{lem incomp}
implies that the matrix is diagonal.
\end{proof}

To prove Theorem \ref{main thm} it is enough to evaluate the diagonal entries of the matrix of $f$ and of the matrices $\mc A_\si, \mc B_\si$
and then check that the diagonal entries satisfy \Ref{wfs}.
The diagonal entries of $\mc A_\si$ are given by part (ii) of Theorem \ref{thm:MO_stab} and the diagonal entries of $\mc B_\si$
are given by part (ii) of Lemma \ref{lem W I sigma}.
Observing that
$c_\bla(\bs \Theta)$ restricted to the fixed point $x_I$ is $c_\bla(\zz_I)$, and that
\bea
\prod_{a<b} \mathop{\prod_{\si(i)\in I_a}}_{\si(j)\in I_b} \Big(
\prod_{i<j} (z_{\si(i)}-z_{\si(j)}) \prod_{i>j} (z_{\si(i)}-z_{\si(j)}-h)\Big)=
\left( \on{sgn}_{\si,I} e^{hor}_{\si,I,-}\right) \cdot e^{ver}_{\si,I,-}
\eea
we obtain that the diagonal entries indeed satisfy \Ref{wfs}. Theorem \ref{main thm} is proved.

\section{Orthogonality}
\label{sec: ortho}

\subsection{Orthogonality on $\F_\bla$}

Consider the bilinear form on $H_{A}^*(\F_\bla)$ defined by
$
(f,g)_{\F_\bla} = \int_{\F_\bla} fg,
$
where the equivariant integral on $\F_\bla$ can be expressed via
localization by
\bea
\int_{\F_\bla} \alpha(\GG;\zz) = \sum_{I\in \II_\lambda} \,\frac{\alpha(\zz_I;\zz)}{\prod_{a<b} \prod_{i\in I_a, j\in I_b} (z_j-z_i)}.
\eea
Note that the denominator is the equivariant Euler class of the tangent space to $\F_\bla$ at the fixed point $x_I$.

Let $\si_0$ be the longest permutation in $S_n$, that is, $\si_0: i\mapsto n+1-i$. It is well known in Schubert calculus that
\bea
( [\bar{\Omega}_{\on{id},J}], [\bar{\Omega}_{\si_0,K}] )_{\F_\bla} = \delta_{J,K},
\eea
where $\bar \Omega_{\si,I}$ is the closure of the Schubert cell $\Omega_{\si,I}$.
In Theorem \ref{thm:Torto} below we will show the $T^*\F_\bla$ version of this orthogonality statement.

\subsection{Orthogonality on $T^*\F_\bla$}

Consider the bilinear form on $H_{T}^*(T^*\F_\bla)$ defined by
\bea
(f,g)_{T^*\F_\bla} = \int_{T^*\F_\bla} fg,
\eea
where the equivariant integral on $T^*\F_\bla$ is {\em defined} via localization by
\bea
\int_{T^*\F_\bla} \alpha(\GG;\zz;h) = \sum_{I\in\II_\lambda}\, \frac{\alpha(\zz_I;\zz;h)}{\prod_{a<b} \prod_{i\in I_a, j\in I_b} (z_j-z_i)(z_i-z_j-h)}.
\eea
Note that the denominator is the equivariant Euler class of the tangent space to $T^*\F_\bla$ at the fixed point $x_I$.
This bilinear form takes values in $\C(\zz;h)$.

\begin{thm} \label{thm:Torto}
We have
\bea
(\St_{\on{id}}(1_J) , \St_{\si_0}(1_K))_{T^*\F_\bla} = \delta_{J,K} \cdot (-1)^{ \dim \F_\bla }.
\eea
\end{thm}

\begin{proof}

We have
\bea
&&
(\St_{\on{id}}(1_J) , \St_{\si_0}(1_K))_{T^*\F_\bla} =
(-1)^{ \dim \F_\bla }\sum_{I\in \II_\bla} \frac{ ( \St_{\on{id}}(1_J) \St_{\si_0}(1_K))|_{x_I} }{R(z_I)Q(z_I)} =
\\
&&
=(-1)^{ \dim \F_\bla } \sum_{I\in\Il} \frac{W_{\on{id},J}(\zz_I;\zz;h) W_{\si_0,K}(\zz_I;\zz,h)}
{R(\zb_I)\,Q(\zb_I)\>c_\bla(\zb_I)^2} = \delta_{J,K} \cdot (-1)^{ \dim \F_\bla }.
\eea
where the first equality is by definition, the second by Theorem \ref{main thm}
and the third by Lemma~\ref{lem:orto}.
\end{proof}

\section{Yangian actions}
\label{sec Yang}

\subsection{Yangian $\Yn$}
\label{sec yangian}

The Yangian $\Yn$ is the unital associative algebra with generators
\,$T_{\ij}\+s$ \,for \,$i,j=1\lc N$, \;$s\in\Z_{>0}$, \,subject to relations
\vvn.3>
\beq
\label{ijkl}
(u-v)\>\bigl[\<\>T_{\ij}(u)\>,T_{\kl}(v)\<\>\bigr]\>=\,
T_{\kj}(v)\>T_{\il}(u)-T_{\kj}(u)\>T_{\il}(v)\,,\qquad i,j,k,l=1\lc N\,,
\vv.3>
\eeq
where
\vvn-.7>
\be
T_{\ij}(u)=\dl_{\ij}+\sum_{s=1}^\infty\,T_{\ij}\+s\>u^{-s}\>.
\vv.2>
\ee

\vsk.2>
The Yangian $\Yn$ is a Hopf algebra with the coproduct
\;$\Dl:\Yn\to\Yn\ox\Yn$ \,given by
\vvn.16>
\;$\Dl\bigl(T_{\ij}(u)\bigr)=\sum_{k=1}^N\,T_{\kj}(u)\ox T_{\ik}(u)$ \,for
\,$i,j=1\lc N$\>. The Yangian $\Yn$ contains, as a Hopf subalgebra, the universal enveloping algebra \>$\Ugln$ of the Lie algebra $\gln$.
The embedding is given by \,$e_{\ij}\mapsto T_{\ji}\+1$, where $e_{\ij}$ are standard standard generators of \,$\gln$.

\vsk.2>
Notice that \,$\bigl[\<\>T_{\ij}\+1,T_{\kl}\+s\<\>\bigr]\>=\,
\dl_{\il}\>T_{\kj}\+s-\dl_{\jk}\>T_{\il}\+s$ \,for \,$i,j,k,l=1\lc N$,
\;$s\in\Z_{>0}$\,, \,which implies that the Yangian \>$\Yn$ is generated by
the elements \,$T_{\ii+1}\+1\>,\>T_{\ioi}\+1$, \;$i=1\lc N-1$, \,and
\,$T_{1,1}\+s$, \;$s>0$.

\subsection{Algebra \,$\ty\>$}
\label{sec tilde Y}

In this section we follow \cite[Section 3.3]{GRTV}.
In formulas of that Section 3.3 we replace $h$ with $-h$.

Let \,$\ty$ be the subalgebra of \,$\Yn\ox\C[h]$ generated over \,$\C\>$
by \,$\C[h]$ and the elements \,$(-h)^{s-1}\>T_{\ij}\+s$ \,for \,$i,j=1\lc N$,
\;$s>0$. Equivalently, the subalgebra \,$\ty$ \,is generated over \,$\C\>$ by
\,$\C[h]$ and the elements \,$T_{\ii+1}\+1\>,\>T_{\ioi}\+1$,
\;$i=1\lc N-1$, \,and \,$(-h)^{s-1}\>T_{1,1}\+s$, \;$s>0$.


For \;$p=1\lc N$, \;$\ib=\{1\leq i_1<\dots<i_p\leq N\}$,
\;$\jb=\{1\leq j_1<\dots<j_p\leq N\}$, define
\vvn-.3>
\be
M_{\ijb}(u) =\sum_{\si\in S_p}(-1)^\si\,
T_{i_1,j_{\si(1)}}(u)\dots T_{i_p,j_{\si(p)}}(u-p+1)\,.
\ee
Introduce the series \,$A_1(u)\lc A_N(u)$, \,$E_1(u)\lc E_{N-1}(u)$,
\,$F_1(u)\lc F_{N-1}(u)$:
\vvn.1>
\begin{gather}
\label{A}
A_p(u)\,=\,M_{\iib}(-u/h)\,=\,
1+\sum_{s=1}^\infty\,(-h)^s\>A_{p,s}\,u^{-s}\,,
\\[3pt]
E_p(u)\,=\,-h^{-1}M_{\jib}(-u/h)\>\bigl(M_{\iib}(-u/h)\bigr)^{-1}\,=\,
\sum_{s=1}^\infty\,(-h)^{s-1}\>E_{p,s}\,u^{-s}\,,
\label{EF}
\\[3pt]
F_p(u)\,=\,-h^{-1}\bigl(M_{\iib}(-u/h)\bigr)^{-1}M_{\ijb}(-u/h)\,=\,
\sum_{s=1}^\infty\,(-h)^{s-1}\>F_{p,s}\,u^{-s}\,,
\notag
\\[-20pt]
\notag
\end{gather}
where  in formulas \Ref{A} and \Ref{EF} we have \,$\ib=\{1\lc p\>\}$\>, \,$\jb=\{1\lc p-1,p+1\>\}$\>.
Observe that \,$E_{p,1}=T_{\pop}\+1$\>, \,$F_{p,1}=T_{\ppo}\+1$
\,and \,$A_{1,s}=T_{1,1}\+s$, so the coefficients of the series
\,$E_p(u)$, \,$F_p(u)$ and \,$h^{-1}(A_p(u)-1\bigr)$ together with \,$\C[h]$
generate \,$\ty$.
In what follows we will describe actions of the algebra \,$\ty$ by using series
\Ref{A}, \Ref{EF}.

\subsection{ $\ty$-action on $(\C^N)^{\otimes n}\otimes\C[\zz;h]$}
\label{sec yang act C^N}

Let $\C[\zz;h]$ act on $(\C^N)^{\otimes n}\otimes\C[\zz;h]$
\vvn.2>
by multiplication. Set
\begin{align}
\label{Lpm}
L(u)\,&{}=\,(u-z_n-h\<\>P^{(0,n)})\dots(u-z_1-h\<\>P^{(0,1)})\,,
\end{align}
where the factors of \,$\C^N\!\ox (\C^N)^{\otimes n}$ are labeled by \,$0,1\lc n$.
\vvn.1>
$L(u)$ is a polynomial in \,$u,\zz,h$ \,with values
in \,$\End(\C^N\!\ox (\C^N)^{\otimes n})$. We consider $L(u)$ as an $N\!\times\!N$ matrix
with \,$\End(V)\ox\C[u;\zz; h]\>$-valued entries \,$L_{\ij}(u)$.

\begin{prop} [Proposition 4.1 in \cite{GRTV}]

The assignment
\vvn-.2>
\beq
\label{pho}
\pho\bigl(T_{\ij}(-u/h)\bigr)\,=\,
L_{\ij}(u)\,\prod_{a=1}^n\,(u-z_a)^{-1}
\vv-.1>
\eeq
defines the action of the algebra \,$\ty$
\vvn.2>
on \,$(\C^N)^{\otimes n}\otimes\C[\zz;h]$\,. Here the right-hand side of \Ref{pho} is a series
in \,$u^{-1}$ with coefficients in \,$\End((\C^N)^{\otimes n})\otimes\C[\zz;h])$\>.
\end{prop}

Under this action, the subalgebra $U(\gln)\subset\ty$ acts on $(\C^N)^{\otimes n}\otimes\C[\zz;h]$
in the standard way: any element \,$x\in\gln$ \,acts as $x^{(1)}\lsym+x^{(n)}$.

\smallskip
The action $\pho$ was denoted in \cite{GRTV} by $\pho^+$.
\smallskip
After the identification
$H_T^*((\XX_n)^A) = (\C^N)^{\otimes n}\otimes \C[\zz;h]$, the action $\pho$ defines
the $\ty$-module structure on $H_T^*((\XX_n)^A)$. This $\ty$-module structure on $H_T^*((\XX_n)^A)$ coincides with
the Yangian module structure on $H_T^*((\XX_n)^A)$ introduced in \cite[Section 5.2.6]{MO}.

\subsection{$H^*_T(\XX_n)$ as a $\ty$-module according to \cite{GRTV}}
\label{sec cohom and Yang}

We define the \>$\ty$-module structure $\rho$ on \;$H_T^*(\XX_n)$ by formulas \Ref{Arho}, \Ref{rho E}, \Ref{rho F}.
Notice that this \>$\ty$-module structure was denoted in \cite{GRTV} by $\rho^-$ and $h$ in \cite{GRTV} is replaced with $-h$.
We define
\;$\rho\bigl(A_p(u)\bigr): H^*_T(T^*\F_\bla)\to H^*_T(T^*\F_\bla)$ by
\vvn.2>
\beq
\label{Arho}
\rho\bigl(A_p(u)\bigr)\>:\>[f]\;\mapsto\,
\Bigl[\>f(\GG;\zb;h)\;\prod_{a=1}^p\,\prod_{i=1}^{\la_p}
\;\Bigl(1-\frac h{u-\ga_{\pci}}\Bigr)\Bigr]\,,
\vv.3>
\eeq
for \,$p=1\lc N$. In particular,
\vvn.4>
\bean
\label{rhoX inf}
\rho(\Xin_i)\>:\>[f]\;\mapsto\,
[\>(\ga_{i,1}\lsym+\<\ga_{i,\la_i})\>f(\GG;\zb;h)]\,,\qquad i=1\lc N\,.
\kern-4em
\eean
Let \;$\aal_1\lc\aal_{N-1}$ \,be simple roots,
\,$\aal_p=(0\lc0,1,\<-\>1,0\lc0)$, with \,$p-1$ first zeros.
We define
\vvn-.6>
\bea
\rho\bigl(E_p(u)\bigr)\<\>:\>H^*_T(T^*\F_{\bla\<\>-\aal_p})\,\mapsto\,H^*_T(T^*\F_\bla)\,,
\eea
\bean
\label{rho E}
\phantom{aaaaa}
\rho\bigl(E_p(u)\bigr)\<\>:\>[f]\;\mapsto\,\biggl[\;
\sum_{i=1}^{\la_p}\;\frac{f(\GG^{\ipi};\zb;h)}{u-\ga_{\pci}}\,\;
\prod_{\satop{j=1}{j\ne i}}^{\la_p}\,\frac1{\ga_{\pcj}-\ga_{\pci}}\;
\prod_{k=1}^{\la_{p+1}\!}\,(\ga_{\pci}-\ga_{p+1,k}-h)\,\biggr]\,,
\eean
\vv-.3>
\bea
\rho\bigl(F_p(u)\bigr)\<\>:\>H^*_T(T^*\F_{\bla\<\>+\aal_p})\,\mapsto\,H^*_T(T^*\F_{\bla})\,,
\eea
\bean
\label{rho F}
\phantom{aaaaa}
\rho\bigl(F_p(u)\bigr)\<\>:\>[f]\;\mapsto\,\biggl[\;
\sum_{i=1}^{\la_{p+1}\!}\;\frac{f(\GG^{\ip};\zb;h)}{u-\ga_{\poi}}\,\;
\prod_{\satop{j=1}{j\ne i}}^{\la_{p+1}\!}\,\frac1{\ga_{\poi}-\ga_{\poj}}\;
\prod_{k=1}^{\la_p}\,(\ga_{p,k}-\ga_{\poi}-h)\,\biggr]\,,
\eean
where
\vvn-.5>
\begin{gather*}
\GG^{\ipi}=\,(\Ga_1\<\>\lsym;\Ga_{p-1}\<\>;\Ga_p-\{\ga_{\pci}\};
\Ga_{p+1}\cup\{\ga_{\pci}\};\Ga_{p+2}\<\>\lsym;\Ga_N)\,,
\\[8pt]
\GG^{\ip}=\,(\Ga_1\<\>\lsym;\Ga_{p-1}\<\>;\Ga_p\cup\{\ga_{\poi}\};
\Ga_{p+1}-\{\ga_{\poi}\};\Ga_{p+2}\<\>\lsym;\Ga_N)\,.
\\[-14pt]
\end{gather*}

\begin{thm} [Theorem 5.10 in \cite{GRTV}]
These formulas define a \>$\ty$-module structure on \;$H^*_T(\XX_n)$.
\end{thm}

The topological interpretation of this action see in \cite[Theorem 5.16]{GRTV}.

The $\ty$-module structure $\rho$ on \;$H^*_T(\XX_n)$ is a Yangian version of representations
of the quantum affine algebra $U_q(\widehat{\gln})$ considered in \cite{Vas1, Vas2}.

\subsection{Stable envelopes and Yangian actions}
As we know, formula \Ref{pho} defines the $\ty$-module structure $\pho$ on
$H^*_T((\XX_n)^A)=(\C^N)^{\otimes n}\otimes \C(\zz;h)$, and formulas \Ref{Arho}, \Ref{rho E}, \Ref{rho F} define
the $\ty$-module structure $\rho$ on $H^*_T(\XX_n)$.

\begin{thm}
\label{thm stab yang}
For the identity element $\on{id}\in S_n$, the map $\St_{\on{id}} : H^*_T((\XX_n)^A) \to H^*_T(\XX_n)$ is a homomorphism
of $\ty$-modules.
\end{thm}

\begin{cor}
\label{Y MO}
The $\ty$-module structure $\rho$ on $H^*_T(\XX_n)$ coincides with the Yangian module structure on
$H^*_T(\XX_n)$ introduced in \cite{MO}.

\end{cor}

\noindent
{\it Proof of Corollary \ref{Y MO}.}\ As we know, the Yangian module structure on $H^*_T((\XX_n)^A)$, introduced in \cite{MO},
coincides with the $\ty$-module structure $\pho$.
In \cite{MO}, the Yangian module structure on $H^*_T(\XX_n)$ is induced from the Yangian module structure on
$H^*_T((\XX_n)^A)$ by the map $\St_{\on{id}} : H^*_T((\XX_n)^A) \to H^*_T(\XX_n)$, see \cite[Sections 4 and 5]{MO}.
Now Theorem \ref{thm stab yang} implies the corollary.
\qed
\smallskip

\subsection{Proof of Theorem \ref{thm stab yang}}
\label{pr st yang}
Define operators \,$\st_1\lc\st_{n-1}$ acting on \<$(\C^N)^{\otimes n}\<$-valued
functions of $\zz,h$ \,by \vvn.2>
\beq
\label{stilde}
\st_i\>f(\zzz,h)=
\frac{(z_i-z_{i+1})\,P^{(\ii+1)}+ h}{z_i-z_{i+1}+ h}\;f(\zzii,h)\,.
\notag
\eeq
\begin{lem} [Lemma 2.3 in \cite{GRTV}]
\label{st+}
The assignment \,$s_i\mapsto\st_i$, \,$i=1\lc n-1$,
\,defines an action of $S_n$.
\qed
\end{lem}

For $I\in\Il$, introduce $\xi_I\in \Czh$ by the formula
\bean
\label{xi}
\xi_I = \sum_{J\in \Il} \frac{W_{\si_0,J}(\zz_I;\zz;h)}{Q(\zz_I)\,c_\bla(\zz_I)} v_J,
\eean
where $\si_0\in S_n$ is the longest permutation and $v_I$ is defined in \Ref{v_I}. Notice that
$\frac{W_{\si_0,J}(\zz_I;\zz;h)}{c_\bla(\zz_I)}$ is a polynomial for every $J$, by Lemma \ref{lem W I sigma}.
Let
\be
D\,=\!\!\prod_{1\leq i<j\leq n}(z_j-z_i-h)\,.
\ee
Define \,$\IMI\in\Il$ \,by $
\IMI\,=\,\bigl(\{1\lc\la_1\}\>,\{\la_1+1\lc\la_1+\la_2\}\>,\;\ldots\;,
\{n-\la_N+1\lc n\}\bigr)\,.
$

\begin{lem} [C.f. Proposition 2.14 in \cite{GRTV}]
\label{lem 2.14}
The elements $\xi_I,I\in \Il$, are unique elements of $(\C^N)^{\otimes n}\otimes\C[\zz;h;D^{-1}]$ such that
$\xi_\IMI=v_\IMI$ \,and $\xi_{s_i(I)}\>=\,\st_i\>\xi_I$
for every \,$I\in\Il$ \,and \,$i=1\lc n-1$\,.
\end{lem}

\begin{proof}
The fact that $\xi_\IMI=v_\IMI$ follows from Lemma \ref{lem W I sigma}.
The property $\xi_{s_i(I)}\>=\,\st_i\>\xi_I$ follows from Lemma \ref{lem W si W}.
\end{proof}

By comparing Lemma \ref{lem 2.14} and \cite[Proposition 2.14]{GRTV} we conclude
that the elements $\xi_I,I\in\Il$, coincide with the elements $\xi_I^+, I\in\Il$, of
Proposition 2.14 in which $h$ is replaced with $-h$.

Consider the map
\bea
\nu=\oplus_{|\bla|=n}\nu_\bla \
: \ \oplus_{|\bla|=n} H_T^*(\XX_\bla)\otimes\C(\zz;h)\ \to \
\Czh,
\eea
where $\nu_\bla$ is defined by the formula
\bea
[f(\bs\Theta;\zz;h)] \mapsto \sum_{I\in\Il} \frac{f(\zz_I;\zz;h)}{R(\zz_I)}\xi_I\,,
\eea
see \cite[Formula (5.9)]{GRTV}.

\begin{lem}
\label{lem nu stab = 1}
We have $\nu\circ\St_{\on{id}} = \on{Id}$.

\end{lem}
\begin{proof}We have
\bea
&&
\nu\circ\St_{\on{id}}(1_I) = \nu\Big(\frac{[W_{\on{id},I}(\bs\Theta;\zz;h)]}{c_\bla(\bs\Theta)}\Big)=
\sum_{J\in\Il} \frac{W_{\on{id},I}(\zz_J;\zz;h)}{R(\zz_J)c_\bla(\zz_J)}\,\xi_J=
\\
&&
\phantom{aaa}
=
\sum_{J,K\in\Il} \frac{W_{\on{id},I}(\zz_J;\zz;h)\,W_{\si_0,K}(\zz_J;\zz;h)}{R(\zz_J)Q(\zz_J)c_\bla(\zz_J)^2}\, v_K = v_I = 1_I\,,
\eea
where the next to the last equality follows from Lemma \ref{lem:orto}.
\end{proof}

Theorem 5.10 in \cite{GRTV} says that $\nu$ is a homomorphism of the $\ty$-module structure $\rho$ on
$H^*_T(\XX_n)\otimes\C(\zz;h)$ to the $\ty$-module structure $\pho$ on
$\Czh$. This proves Theorem \ref{thm stab yang}.

\section{Dynamical Hamiltonians and quantum multiplication}
\label{Bethe subalgebras sec}

\subsection{Dynamical Hamiltonians}
\label{sec dyn hams}
Assume that \,$\kkk$ \>are distinct numbers. Define the elements
\,$\Xk_1\lc\Xk_N\in\ty$ by the rule
\vvn.3>
\bea
\label{X}
\Xk_i\, {}=\,-h\>T_{\ii}\+2+\>\frac h2\,T_{\ii}\+1\>\bigl(\<\>T_{\ii}\+1-1\bigr)-\>
h\>\sum_{\satop{j=1}{j\ne i}}^N\,\frac{\kk_j}{\kk_i-\kk_j}\,G_{\ij}\,,
\notag
\eea
where  $G_{\ij}=\,T_{\ij}\+1\>T_{\ji}\+1\<-T_{\jj}\+1=\,
T_{\ji}\+1\>T_{\ij}\+1\<-T_{\ii}\+1\>.$
Notice that
\vvn.1>
\be
T_{\ii}\+1=\>e_{\ii}\,,\qquad
G_{\ij}\>=\,e_{\ij}\>e_{\ji}\<-e_{\ii}\>=\,e_{\ji}\>e_{\ij}\<-e_{\jj}\>.
\ee
By taking the limit $\kk_{i+1}/\kk_i\to0$ \,for all \,$i=1\lc N-1$,
we define the elements \,$\Xin_1\lc\Xin_N\in\ty$,
\vvn-.1>
\bea
\label{X8}
\Xin_i {}=\,-h\>T_{\ii}\+2+\>\frac h2\,e_{\ii}\>\bigl(\<\>e_{\ii}-1\bigr)
+h\>(G_{i,1}\<\lsym+\<\>G_{\ii-1})\,,
\\[-14pt]
\notag
\eea
see \cite{GRTV}. We call the elements
\vvn.1>
\,$\Xk_i\<,\> \Xin_i$, \,$i=1\lc N$, the {\it dynamical Hamiltonians\/}.
Observe that
\vvn-.3>
\bea
\label{XX}
\Xk_i\,=\,\Xin_i-\>h\>\sum_{j=1}^{i-1}\,\frac{\kk_i}{\kk_i-\kk_j}\,G_{\ij}
\>-\>h\!\sum_{j=i+1}^n\frac{\kk_j}{\kk_i-\kk_j}\,G_{\ij}\,.
\eea
Given \,$\bla=(\la_1\lc\la_N)$, \,set
\vvn.1>
\;$G_{\bla\<\>,\<\>\ij}\>=\,e_{\ji}\>e_{\ij}$ \,for \,$\la_i\ge\la_j$\>
\,and
\;$G_{\bla\<\>,\<\>\ij}=\>e_{\ij}\>e_{\ji}$ \,for \,$\la_i<\la_j$\>.
Define the elements
\,$X^q_{\bla\<\>,1}\,\lc X^q_{\bla\<\>,\<\>N}\<\in\ty$,
\bea
\label{Xkm}
X^q_{\bla\<\>,\<\>i}\>=\,\Xin_i-
\>h\>\sum_{j=1}^{i-1}\,\frac{\kk_i}{\kk_i-\kk_j}\,G_{\bla\<\>,\<\>\ij}\>-
\>h\!\sum_{j=i+1}^n\frac{\kk_j}{\kk_i-\kk_j}\,G_{\bla\<\>,\<\>\ij}\,.
\eea

Let $\kappa\in\C^\times$. The formal differential operators
\vvn-.4>
\beq
\label{dyneq}
\nabla_{\qq,\kp,i}\>=\,\kp\,\ddk_i\>-\>\Xk_i,\qquad i=1\lc N,
\eeq
pairwise commute and, hence, define a flat connection for any \,$\ty$-module, see \cite{GRTV}.

\begin{lem}[Lemma 3.5 in \cite{GRTV}]
\label{flat+-}
The connection \;\;$\nabla_{\<\bla,\qq,\kp}\<$ defined by
\vvn.2>
\bea
\label{nablapm}
\nabla_{\<\bla\<\>,\qq,\kp,\<\>i}=\,\kp\,\ddk_i\>-\>X^q_{\bla\<\>,\<\>i}\,,
\vv-.1>
\eea
$i=1\lc N$, is flat for any \,$\kp$.
\end{lem}
\begin{proof}
The connection \;\;$\nabla_{\<\bla\<\>,\qq,\kp,\<\>i}\<$ is gauge
equivalent to connection \Ref{dyneq},
\vvn.4>
\bea
\label{gauge}
\nabla_{\<\bla\<\>,\qq,\kp,\<\>i}=\,(\Upsilon_\bla)^{-1}\;\nabla_{\qq,\kp,i}\;\Upsilon_\bla\,,
\qquad \Upsilon_\bla\>=\prod_{1\le i<j\le n}\!(1-\kk_j/\kk_i)
^{\<\>h\>\eps_{\bla\<\>,\ij}/\<\kp}\,,
\vv.2>
\eea
where \;$\eps_{\bla\<\>,\ij}=\>e_{\jj}$ \,for \,$\la_i\ge\la_j$\>,
\,and \;$\eps_{\bla\<\>,\ij}=\>e_{\ii}$ \,for \,$\la_i<\la_j$\>.
\end{proof}

Connection \Ref{dyneq} was introduced in \cite{TV2}, see also Appendix B in \cite{MTV1}, and
is called the {\it trigonometric dynamical connection\/}. Later the definition was extended
from $\frak{sl}_N$ to other simple Lie algebras in \cite{TL} under the name
of the trigonometric Casimir connection.

The trigonometric dynamical connection is
defined over \;$\C^N$ with coordinates \,$\kkk$, it has singularities
at the union of the diagonals \,$\kk_i=\kk_j$. In the case of a tensor product
of evaluation \,$\ty$-modules, the trigonometric dynamical connection commutes
with the associated \qKZ/ difference connection, see \cite{TV2}. Under
the \,$(\gln,\>\gl_{\>n})$ \,duality, the trigonometric dynamical connection
and the associated \qKZ/ difference connection are respectively identified
with the trigonometric \KZ/ connection and the dynamical difference connection,
see \cite{TV2}.

\subsection{\qKZ/ difference connection}
\label{sec qkz}
Recall the $\ty$-action $\pho$ on $(\C^N)^{\otimes n}\otimes\C[\zz;h]$ introduced in
Section \ref{sec yang act C^N}. \,Let
\vvn.1>
\be
R^{(\ij)}(u)\,=\,\frac{u-h\<\>P^{(\ij)}}{u-h}\;,\qquad
i,j=1\lc n\,,\quad i\ne j\,.\kern-3em
\vv.3>
\ee
For $\kp\in\C^\times$, define operators \,$\Ko_1\lc\Ko_n\in\End( (\C^N)^{\otimes n})\ox\C[\zz;h]$\>,
\vvn.4>
\begin{align*}
\Ko_i(\qq;\kp)\>&{}=\,
R^{(\ii-1)}(z_i\<-z_{i-1})\,\dots\,R^{(i,1)}(z_i\<-z_1)\,\times{}
\\[2pt]
& {}\>\times\<\;\kk_1^{e_{1,1}^{(i)}}\!\dots\,\kk_N^{e_{N,N}^{(i)}}\,
R^{(i,n)}(z_i\<-z_n\<-\kp\<\>)\,\dots\,R^{(\ii+1)}(z_i\<-z_{i+1}\<-\kp\<\>)\,.
\end{align*}
Consider the difference operators \,$\Kh_{\qq,\kp,1}\lc\Kh_{\qq,\kp,n}$
acting on \,$ (\C^N)^{\otimes n}\<$-valued
\vvn.4>
functions of \>$\zz,\qq,h$\>,
\bea
\label{K_i}
\Kh_{\qq,\kp,i}\>f(\zzz,h)\,=\,K_i(\qq;\kp)\,f(\zzip)\,.
\vv.3>
\eea

\begin{thm}[\cite{FR}]
\label{thmfr}
The operators \;$\Kh_{\qq,\kp,1}\lc\Kh_{\qq,\kp,n}$ pairwise commute.
\end{thm}

\begin{thm}[\cite{TV2}]
\label{thm qkz}
The operators \;$\Kh_{\qq,\kp,1}\lc\Kh_{\qq,\kp,n}$,
\,$\pho(\nabla_{\<\bla\<\>,\qq,\kp,\<\>i})\lc\pho(\nabla_{\<\bla\<\>,\qq,\kp,\<\>i})$ \,pairwise commute.
\end{thm}

The commuting difference operators $\Kh_{\qq,\kp,1}\lc\Kh_{\qq,\kp,n}$ define the {\it rational \qKZ/
difference connection}.
Theorem \ref{thm qkz} says that the rational \qKZ/ difference connection
commutes with the trigonometric dynamical connection.

\subsection{Dynamical Hamiltonians $X^q_{\bla,i}$ on $H^*_T(T^*\F_\bla)$} Recall the $\ty$-module structure
$\rho$ defined on $H^*_T(\XX_n) = \oplus_{|\bla|=n}H^*_T(T^*\F_\bla)$ in Section \ref{sec cohom and Yang}.
For any $\bs\mu=(\mu_1,\dots,\mu_N)\in \Z^N_{\geq 0}$, $|\bs\mu|=n$, the
action of the dynamical Hamiltonians $X^q_{\bs\mu,i}$ preserve each of $H^*_T(T^*\F_\bla)$.

\begin{lem}
For any $\bla$ and $i=1,\dots,n$, the restriction of $X^q_{\bla,i}$ to
$H^*_T(T^*\F_\bla)$ has the form:
\bean
\label{dyn bla}
&&
\phantom{aaa}
\\
\rho(X^{q}_{\bla,i})
&=&
(\gamma_{i,1}+\dots+\ga_{i,\la_i})
-\>h\>\sum_{j=1}^{i-1}\,\frac{\kk_i}{\kk_i-\kk_j}\,\rho(G_{\bla\<\>,\<\>\ij})
\>-\>h\!\sum_{j=i+1}^n\frac{\kk_j}{\kk_i-\kk_j}\,\rho(G_{\bla\<\>,\<\>\ij}) =
\notag
\\
&=& (\gamma_{i,1}+\dots+\ga_{i,\la_i})
-\>h\>\sum_{j=1}^{i-1}\,\frac{\kk_i}{\kk_i-\kk_j}\,\rho(e_{j,i}e_{i,j})
\>-\>h\!\sum_{j=i+1}^n\frac{\kk_j}{\kk_i-\kk_j}\,\rho(e_{i,j}e_{j,i})\, + C,
\notag
\eean
where $(\gamma_{i,1}+\dots+\ga_{i,\la_i})$ denotes the operator of multiplication by the cohomology class
$\gamma_{i,1}+\dots+\ga_{i,\la_i}$, the operator
$C$ is a scalar operator on $H^*_T(T^*\F_\bla)$, and for any $i\ne j$ the element
$\rho(G_{\bla\<\>,\<\>\ij})$ annihilates the identity element $1_\bla\in H^*_T(T^*\F_\bla)$.

\end{lem}

\begin{proof} The first equality in \Ref{dyn bla} follows from \Ref{rhoX inf}.
The operator $C$ is scalar since
$e_{j,i}e_{i,j}-G_{\bla\<\>,\<\>\ij}$ and
$e_{i,j}e_{j,i}-G_{\bla\<\>,\<\>\ij}$ lie in the Cartan subalgebra and act on
$H^*_T(T^*\F_\bla)$ as scalars. By Theorem \ref{thm stab yang}, in order to show that
$\rho(G_{\bla\<\>,\<\>\ij})$ annihilates the identity element $1_\bla$ it is enough to show that
$\pho(G_{\bla\<\>,\<\>\ij})$ annihilates the element $\nu(1_\bla) = \sum_{I\in\Il}\frac 1{R(\zz_I)}\xi_I$
and that is the statement of \cite[Lemma 2.20]{GRTV}, see also \cite{RTVZ}.
\end{proof}

\subsection{Quantum multiplication by divisors on $H^*_T(T^*\F_\bla)$} In \cite{MO}, the quantum multiplication
by divisors on $H^*_T(T^*\F_\bla)$ is described. The fundamental equivariant cohomology classes of divisors
on $T^*\F_\bla$ are linear combinations of $D_{i}=\gamma_{i,1}+\dots+\ga_{i,\la_i}$, $i=1,\dots,N$.
The quantum multiplication $D_{i}*_{\tilde\qq}$ depends on parameters $\tilde \qq=(\tilde q_1,\dots,\tilde q_N)$.

\begin{thm}[Theorem 10.2.1 in \cite{MO}]
\label{MO main} For $i=1,\dots,N$, the quantum multiplication by $D_i$ is given by the formula:
\bean
\label{q mult}
&&
\phantom{aaa}
\\
D_i*_{\tilde\qq}
&=& (\gamma_{i,1}+\dots+\ga_{i,\la_i})
+\>h\>\sum_{j=1}^{i-1}\,\frac{\tilde q_j/\tilde q_i}{1-\tilde q_j/\tilde q_i}\,\rho(e_{j,i}e_{i,j})
\>-\>h\!\sum_{j=i+1}^n\frac{\tilde q_i/\tilde q_j}{1-\tilde q_i/\tilde q_J}\,\rho(e_{i,j}e_{j,i})\, + C,
\notag
\eean
where $C$ is a scalar operator on $H^*_T(T^*\F_\bla)$ fixed by the requirement that
the purely quantum
part of $D_i*_{\tilde\qq}$ annihilates the identity $1_\bla$.

\end{thm}


\begin{cor}
\label{cor qm = dh}
For $i=1,\dots,N$, the operator $D_i*_{\tilde\qq}$
of quantum multiplication by $D_i$ on $H^*_T(T^*\F_\bla)$
equals the action $\rho(X^{q}_{\bla,i})$ on $H^*_T(T^*\F_\bla)$
of the dynamical Hamiltonian $X^{q}_{\bla,i}$ if we put
$(q_1,\dots,q_N)$ $ = (\tilde q_1^{\,-1},\dots, \tilde q_N^{\,-1})$.

\end{cor}

The quantum connection $\nabla_{\on{quant},\bla,\tilde\qq,\kp}$ on $H^*_T(T^*\F_\bla)$ is defined by the formula
\bean
\label{q conn}
\nabla_{\on{quant},\bla,\tilde \qq,\kp, i}\,=\,\kp\,\tilde q_i\frac{\der}{\der \tilde q_i}\>-\>D_i*_{\tilde\qq}\,,\qquad i=1\lc N,
\vv-.1>
\eean
where $\kp\in\C^\times$ is a parameter of the connection, see \cite{BMO}.
By Corollary \ref{cor qm = dh}, we have
\bean
\label{quanT}
\nabla_{\on{quant},\bla,\tilde \qq,\kp, i} = \rho(\nabla_{\bla,\tilde q_1^{\,-1},\dots, \tilde q_N^{\,-1},-\kp}),
\qquad i=1,\dots,N.
\eean

\smallskip

By Theorem \ref{thm qkz}, the difference operators
\bea
\St_{\on{id}}\circ\Hat\Ko_{\tilde q_1^{\,-1},\dots, \tilde q_N^{\,-1},-\kp,1} \circ \nu,
\quad
\dots,
\quad
\St_{\on{id}}\circ\Hat\Ko_{\tilde q_1^{\,-1},\dots, \tilde q_N^{\,-1},-\kp,n} \circ\nu
\eea
and the differential operators $\nabla_{\on{quant},\bla,\tilde\qq,\kp,1},\dots,\nabla_{\on{quant},\bla,\tilde\qq,\kp,N}$
pairwise commute. The difference operators form
the rational \qKZ/ difference connection on $H^*_T(T^*\F_\bla)$. This difference connection is discussed in \cite{MO}
under the name of the shift operators.

\section{Proofs of lemmas on weight functions}
\label{proofs}

\subsection{Proof of Lemma \ref{lem W_I}}

Parts (ii-iv) of Lemma \ref{lem W_I} are proved by inspection of the definition of weight
functions.

For $N=2$, part (i) of Lemma \ref{lem W_I} is proved in \cite{RTV}, see Lemma 3.6 and Theorem 4.2 in \cite{RTV}.
The general case $N>2$ is proved as follows.

Let $U_{I,\bs\si}(\TT;\zz;h)$ be the term
in the symmetrization in \Ref{hWI-} obtained by permuting the variables $t^{(i)}_a$
by an element $\bs \si\in S_{\la^{(1)}}\times\dots\times S_{\la^{(N-1)}}$. Set
\bea
U_{I,J,\bs \si}(\zz;h) = U_{I,\bs\si}(\zz_J;\zz;h) .
\eea
We show that each term $U_{I,J,\bs \si}(\zz;h)$ is divisible by $c_\bla(\zz_J)$.

Recall that $
\cup_{b=1}^a I_a = I^{(a)} = \{ i^{(a)}_1<\dots<i^{(a)}_{\la^{(a)}} \}.
$
Similarly, let
$
\cup_{b=1}^a J_a = J^{(a)} = \{ j^{(a)}_1<\dots<j^{(a)}_{\la^{(a)}} \}.
$

The substitution $\TT=\zz_J$ implies $t^{(N-1)}_a=z_{j^{(N-1)}_a}$.
Denote by $\si$ the component of $\bs\si$ in the last factor $S_{\la^{(N-1)}}$.
Consider the factor $f_{c,d} = z_{j^{(N-1)}_c} - z_{j^{(N-1)}_d} - h$
in $c_\bla(\zz_J)$ for $c\ne d$.

Let $a=\si^{-1}(c)$. If $j^{(N-1)}_d < i^{(N-1)}_a$, then $f_{c,d}$ divides
$U_{I,J,\bs\si}$ due to the factor $t^{(N-1)}_{\si(a)} - z_{j^{(N-1)}_d} - h$
in $U_{I,\bs\si}$.

Let $b=\si^{-1}(d)$. If $i^{(N-1)}_a \le j^{(N-1)}_{\si(b)} = j^{(N-1)}_d \le i^{(N-1)}_b$,
then $a<b$, because $i^{(N-1)}_a \le i^{(N-1)}_b$ and $a\ne b$.
Then $f_{c,d}$ divides $U_{I,J,\bs\si}$ due to the factor
$t^{(N-1)}_{\si(a)} - t^{(N-1)}_{\si(b)} - h$ in $U_{I,\bs\si}$.

If $i^{(N-1)}_b < j^{(N-1)}_d$, then $U_{I,J,\bs\si}=0$ due to the factor
$t^{(N-1)}_{\si(a)} - z_{j^{(N-1)}_d} $ in $U_{I,\bs\si}$.

Once the substitution $t^{(N-1)}_a=z_{j^{(N-1)}_a}$ is done,
the consideration of the factors
$
z_{j^{(N-2)}_c} - z_{j^{(N-2)}_d} - h
$
in $c_\bla(\zz_J)$ is similar.

\subsection{Proof of Lemma \ref{lem W si W}}

It is enough to prove that for $I\in \Il$, $i=1,\dots,n-1$, we have
\bean
\label{s W}
W_{s_i,I}(\TT;\zz;h) =
\frac{z_{i}-z_{i+1}}{z_i-z_{i+1}+h} W_{\on{id},I}(\TT;\zz;h) +
\frac h{z_{i}-z_{i+1}+h} W_{\on{id}, s_i(I)}(\TT;\zz;h).
\eean
Moreover, it is straightforward to see from \Ref{hWI-} that it suffices to prove
relations \Ref{s W} for \,$n=2$, $i=1$, and the following two cases for $I$.
The first case is $I=(I_1,\dots,I_N)$, where $I_1=\{1,2\}$, and $I_2,\dots,I_N$ are empty.
The second case is $I=(I_1,\dots,I_N)$, where
$I_1=\{1\}$, $I_2=\{2\}$, and $I_3,\dots,I_N$ are empty.
In each of the two cases, formula \Ref{s W} is proved by straightforward verification.
All other cases of formula \Ref{s W} can be deduced from these two by picking up a suitable subexpression
and an appropriate change of notation.

\subsection{Proof of Lemma \ref{lem:orto}}
In addition to vectors $\xi_I,I\in\Il,$ defined in \Ref{xi}, we introduce the vectors
\bean
\label{xi op}
\xi_{\si_0,I} = \sum_{J\in \Il} \frac{W_{\on{id},J}(\zz_I;\zz;h)}{Q(\zz_I)\,c_\bla(\zz_I)} \,v_J.
\eean
Let $\mc S$ be the $\C(\zz;h)$-bilinear form on $(\C^N)^{\otimes n}\otimes \C(\zz;h)$ such that the basis $(v_J)$
is orthonormal. Then the statement of the lemma is equivalent to the statement
\bea
\mc S(\xi_I,\xi_{\si_0,J}) = \delta_{I,J}\,\frac{R(\zz_I)}{Q(\zz_I)},
\eea
which is the statement of \cite[Theorem 2.18]{GRTV}.

\end{document}